\documentclass{preprint-raic}
\usepackage{enumerate}
\usepackage{mathrsfs}
\usepackage{float}

\AUTHORS{
  Martin~Rai\v{c}%
    \footnote{%
      University of Ljubljana, Faculty of Mathematics and Physics, Jadranska 19, SI--1000 Ljubljana, Slovenia, and
      Institute of Mathematics, Physics and Mechanics, Jadranska 19, SI--1000 Ljubljana, Slovenia.
      \EMAIL{martin.raic@fmf.uni-lj.si}%
   }
}

\TITLE{%
  A COMPLETION OF COUNTEREXAMPLES TO THE CLASSICAL CENTRAL LIMIT THEOREM FOR PAIR- AND
  TRIPLEWISE INDEPENDENT AND IDENTICALLY DISTRIBUTED RANDOM VARIABLES
} 

\SHORTTITLE{Counterexample to CLT for triplewise independent r.v.'s}


\AMSSUBJ{60F05} 
\AMSSUBJSECONDARY{62E20} 

\ABSTRACT{%
  By the Lindeberg--Lévy central limit theorem, standardized partial sums of a sequence of
  mutually independent and identically distributed random variables converge in law to
  the standard normal distribution. It is known that mutual independence cannot be relaxed
  to pairwise and even not triplewise independence. Counterexamples have been constructed
  for most marginal distributions: a recent construction works under a condition which
  excludes certain probability distributions with atomic parts, in particular discrete
  distributions in the ``general position.'' In the present paper, we show that this
  condition can be lifted: for any probability distribution $ F $ on the real line,
  which has finite variance and is not concentrated in a single point, there exists a
  sequence of triplewise independent random variables with distribution $ F $, such that
  its standardized partial sums converge in law to a distribution which is not normal.
  There is also scope for extension to $ k $-tuplewise independence.
}

\expandafter\def\csname citep\endcsname{\cite}

\newcommand{\RR}{\mathbb{R}}

\newcommand{\sth}{\mathrel{;}}

\DeclareMathOperator{\Pp}{\mathbb{P}}
\DeclareMathOperator{\Ee}{\mathbb{E}}
\DeclareMathOperator{\Var}{Var}

\DeclareMathOperator{\One}{\mathds 1}

\begin{document}

\section{Introduction}
\label{sc:Intr}

For a sequence of mutually independent and identically distributed random variables
$ X_1, X_2, \ldots $ with $ \Ee X_1 = \mu $ and $ \Var(X_1) = \sigma^2 $, where
$ 0 < \sigma < \infty $, it is known that the standardized partial sums
\begin{equation}
\label{eq:Sn}
 S_n := \frac{1}{\sigma \sqrt n} \biggl( \sum_{k=1}^n X_k - n \mu \biggr)
\end{equation}
converge in law to the standard normal distribution: this is known as the
Lindeberg--Lévy central limit theorem \cite{Lind0,Levy}.
It is also known that mutual independence can in general not be relaxed to the weaker
pairwise independence, nor can it even be relaxed to triplewise independence.
In general, $ K $-tuplewise independence is defined as follows:

\begin{definition}
Let $ K \in \{ 2, 3, 4, \ldots \} $. An indexed family of random variables
$ X_i $, $ i \in I $, is \emph{$ K $-tuplewise independent} if the random
variables $ X_{i_1}, X_{i_2}, \ldots, X_{i_K} $ are mutually independent
for any $ K $-tuple of distinct indices $ i_1, i_2, \ldots, i_K $.
\end{definition}

Various counterexamples have been constructed. Avanzi et al.\ \cite{CE2}
provide a survey of earlier constructions and construct a family of counterexamples
for the pairwise independence. Boglioni Beaulieu et al.\ \cite{CE3} modify the latter
construction to one which is based on a suitable sequence of graphs, each
graph giving a family of $ K $-tuplewise independent and identically distributed
random variables. The random variables obtained from all graphs can be arranged
into an array, each graph giving one row. They provide an increasing sequence of
graphs giving triplewise independent rows and standardized row sums converging in
law to a variance-gamma distribution, which is not normal: see Subsection~4.1 ibidem.
From that array, a sequence can be extracted, such that its standardized partial
sums do not converge to a normal distribution because it has a subsequence which
converges to the variance-gamma distribution. Under some additional conditions,
it can be shown that the entire sequence converges to the same distribution:
see Lemma~\ref{lm:FillGaps}.

Both above-mentioned constructions allow for a broad choice for the (common)
distribution of the summands $ X_1, X_2, \ldots $\spacefactor=3000{}
Indeed, taking $ W $ to be a generic random variable with this distribution, both
constructions work under the following condition quoted below:

\begin{condition}
\label{cn:l-1}
There exists a Borel set $ A \subseteq \RR $, such that:
\begin{list}{\textlegacybullet}{\topsep 1ex \parsep 0ex \itemsep 1ex}
\item $ \Pp(W \in A) = \ell^{-1} $ for some integer $ \ell \ge 2 $;
\item $ \Ee(W \mid W \in A) \ne \Ee(W \mid W \notin A) $.
\end{list}
\end{condition}

Although this restriction is relatively mild, not all probability
distributions on the real line fit it. In particular, discrete
distributions in the ``general position'' are excluded, concretely
any discrete distribution which is non-trivial and with point
probabilities summing up only to $ 0 $, $ 1 $ or an irrational number;
compare Remark~2 in \cite{CE2}.

However, we show that Condition~\ref{cn:l-1} can be lifted: the constructions
provided by \cite{CE2,CE3} can be adapted so that they allow for any distribution
on the real line which makes sense. Indeed, instead of Condition~\ref{cn:l-1},
we only need that the distribution of $ W $ can be represented as a suitable
mixture of two distributions. The following assertion states that this is true
for all distributions which make sense (and we only need the case
$ \tau = \ell^{-1} $). We defer the proof to Section~\ref{sc:Mix}.

\begin{proposition}
\label{pr:MixDiffExp}
For each $ \tau \in (0, 1) $ and any real-valued random variable $ W $ with
finite expectation, which is not almost surely constant, there exist
real-valued random variables $ U $ and $ V $ with different expectations,
such that
\begin{equation}
\label{eq:MixDiffExp}
 \Pp(W \in C) = (1 - \tau) \Pp(U \in C) + \tau \Pp(V \in C)
\end{equation}
for all Borel sets $ C \subseteq \RR $. Moreover, if $ W $ has finite variance,
$ U $ and $ V $ can be chosen to have finite variances, too.
\end{proposition}

Based on the argument given in \cite{CE2,CE3} extended by
Proposition~\ref{pr:MixDiffExp}, we are able to complete the family
of counterexamples to the central limit theorem for triplewise independent summands,
as specified in the following result:

\begin{theorem}
\label{th:Seq3}
For any random variable $ W $ on the real line with finite variance, which is not
almost surely constant, there exists a sequence $ X_1, X_2, \ldots $
of triplewise independent random variables, which follow the same distribution as $ W $,
such that the standardized partial sums $ S_n $ defined as in \eqref{eq:Sn} converge
in law to a probability distribution which is not normal.
\end{theorem}

We defer the proof to the end of Section~\ref{sc:Ada}, where we give an outline
of the arguments given in \cite{CE2,CE3}, exposing the point where
Proposition~\ref{pr:MixDiffExp} is applied. Notice that the latter is not related
to the dependence structure of the summands, which depends on a sequence of
graphs. So far, sequences leading to counterexamples for pair- and triplewise
independence have been constructed. In future, it may turn out that higher degree
of tuplewise independence can also be covered: see the discussion in Chapter~5 of
\cite{CE3}. As stated in Corollary~\ref{co:CESeq}, this would automatically extend
Theorem~\ref{th:Seq3}, preserving the generality of the distribution of the
summands.

\section{Adaptation of construction}
\label{sc:Ada}

As mentioned in the Introduction, the construction provided by
Boglioni Beaulieu et al.\ \cite{CE3} starts with a sequence of undirected graphs
$ G_1, G_2, \ldots $\spacefactor=3000{} In order to provide $ K $-tuplewise
independence, all these graphs must be of girth at least $ K + 1 $, that
is, there must be no cycles of length $ K $ or less.

For a graph $ G $, denote as usual by $ V(G) $ its vertex set and by $ E(G) $ its
edge set. We work with abstract edges, assuming that each edge is assigned its
two endpoints. In the sequence $ G_1, G_2, \ldots $, we assume that this assignment
is consistent for all graphs in it, that is, any edge $ k \in E(G_m) \cap E(G_n) $
has the same endpoints in both $ G_m $ and $ G_n $.

Following Boglioni Beaulieu et al.\ \cite{CE3} (altering the notation
to some extent), choose $ \ell \in \{ 2, 3, 4, \ldots \} $.
Define $ \mathcal V := \bigcup_{m=1}^\infty V(G_m) $ and
$ \mathcal E := \bigcup_{m=1}^\infty E(G_m) $. For each vertex $ i \in \mathcal V $, consider
a random variable $ M_i $ distributed uniformly over $ \{ 1, 2, \ldots, \ell \} $,
letting all random variables $ M_i $, $ i \in \mathcal V $, be mutually independent.
For each edge $ k \in \mathcal E $ with endpoints $ i $ and $ j $, define $ D_k := 1 $ if
$ M_i = M_j $ and $ D_k := 0 $ otherwise. Since each graph $ G_m $ has girth at least
$ K + 1 $, the family $ D_k, k \in E(G_m) $, is $ K $-tuplewise independent.
Denoting by $ n_m $ the number of edges of $ G_m $, let
\[
 \Xi^*_m := \sum_{k \in E(G_m)} D_k
 \, , \kern 1.5em
 \xi^*_m := \frac{\Xi^*_m - n_m \ell^{-1}}{\sqrt{n_m \ell^{-1} (1 - \ell^{-1})}}
 \, .
\]
Now choose two generic real-valued random variables $ U $ and $ V $ with finite
variances. Let $ W $ be a random variable with distribution being a mixture
of the distributions of $ U $ and $ V $: more precisely,
\begin{equation}
\label{eq:Mixl-1}
 \Pp(W \in C) = (1 - \ell^{-1}) \Pp(U \in C) + \ell^{-1} \Pp(V \in C)
\end{equation}
for all Borel sets $ C \subseteq \RR $. Next, for each edge $ k \in \mathcal E $, consider
random variables $ U_k $ and $ V_k $ following the same distribution as
$ U $ and $ V $, respectively. Choose the random variables $ U_k $ and
$ V_k $, $ k \in \mathcal E $, to be all mutually independent as well as
independent of the random variables $ M_i $, $ i \in \mathcal V $. Letting
\begin{equation}
\label{eq:XUVD}
 X_k := \left\{ \begin{array}{cl}
  U_k & \> ; \> D_k = 0
 \\
  V_k & \> ; \> D_k = 1
 \end{array} \right.
\end{equation}
and fixing $ m $, observe that the random variables $ X_k $, $ k \in E(G_m) $,
are $ K $-tuplewise independent and follow the same distribution as $ W $.

The constructions in the papers \cite{CE2,CE3} start with $ W $ and a Borel set
$ A \subseteq \RR $, letting $ U $ and $ V $ to follow the conditional distributions
of $ W $ given $ A $ and $ A^c $, respectively. This gives rise to Condition~\ref{cn:l-1}.
However, there is no need to choose $ U $ and $ V $ this way: all that suffices for the
continuation and desired properties of the construction, in particular Theorem~1 in
\cite{CE2} and Theorem~3.1 in \cite{CE3}, is the relationship \eqref{eq:Mixl-1}: the
latter (along with the observation that $ X_k $ follow the same distribution as
$ W $) corresponds to Formula~$(2.9)$ in \cite{CE2} and Formula~$(2.10)$ in \cite{CE3}.
Along with Formula~$(2.8)$ in \cite{CE2} and Formula~$(2.9)$ in \cite{CE3}
(which both correspond to \eqref{eq:XUVD}) and the independence properties
of the random variables $ U_k $ and $ V_k $, $ k \in E(G_m) $, this is all
that is used in the proofs of Theorem~1 in \cite{CE2} and Theorem~3.1 in \cite{CE3}.
This proves the following modification of Theorem~3.1 in \cite{CE3}:

\begin{theorem}
\label{th:xiS}
With $ U $, $ V $, $ W $, $ G_m $, $ X_k $ and $ \xi^*_m $ as above,
let $ \mu = \Ee W $ and $ \sigma^2 = \Var(W) $; assume that $ 0 < \sigma < \infty $.
Provided that there exists a random variable $ Y $, such that
\[
 \xi^*_m \xrightarrow[n \to \infty]{\mathrm{law}} Y \, ,
\]
the standardized sums
\[
 S^*_m := \frac{\sum_{k \in E(G_m)} X_k - n_m \mu}{\sigma \sqrt{n_m}}
\]
converge in law to the random variable
\[
 S^{(\ell)} := \sqrt{1 - r^2} \, Z + r \, Y \, ,
\]
where $ Z $ a standard normal random variable, independent of $ Y $, and where\newline
$ r := \sqrt{\ell^{-1}(1 - \ell^{-1})} \, (\Ee V - \Ee U)/\sigma $.
\qed
\end{theorem}

\begin{remark}
\label{rk:Deconv}
If $ r > 0 $, then $ S^{(\ell)} $ is normal if and only if $ Y $ is normal.
\end{remark}

However, by Proposition~\ref{pr:MixDiffExp}, each real-valued random variable $ W $
with finite variance, which is not almost surely constant, admits random variables
$ U $ and $ V $ with finite variances and different expectations, such that
\eqref{eq:Mixl-1} is satisfied for all Borel sets $ C \subseteq \RR $. Recalling
Remark~\ref{rk:Deconv}, we have now proved the following assertion.

\begin{corollary}
\label{co:CEArr}
Let $ W $ be a real-valued random variable with expectation $ \mu $ and
variance $ \sigma^2 $; assume that $ 0 < \sigma < \infty $. Let $ G_m $,
$ n_m $ and $ \xi^*_m $, $ m = 1, 2, 3, \ldots $, be as above. Suppose that all graphs
$ G_m $ have girth at least $ K + 1 $ and that the random variables
$ \xi^*_m $ converge in law to a probability distribution which is not normal.
Then there exist random variables $ X_k $, $ k \in \mathcal E $, with the following
properties:
\begin{list}{\textlegacybullet}{\topsep 1ex \parsep 0ex \itemsep 1ex\labelwidth \leftmargini}
\item For each $ k \in \mathcal E $, $ X_k $ has the same distribution as $ W $.
\item For each $ m = 1, 2, 3, \ldots $, the family $ X_k $, $ k \in E(G_m) $,
is $ K $-tuplewise independent.
\item The standardized sums $ S^*_m = \frac{\sum_{k \in E(G_m)} X_k - n_m \mu}{\sigma \sqrt{n_m}} $
converge in law to a probability distribution which is not normal.
\end{list}
\qed
\end{corollary}

The preceding assertion allows us to construct counterexamples to the central limit
theorem in terms of \emph{arrays} of random variables, each graph giving one row.
On the other hand, the central limit theorem is originally formulated in terms of
\emph{sequences}. The latter can also be constructed if the grapgs $ G_m $ form an
increasing sequence in the sense that
$ V(G_1) \subseteq V(G_2) \subseteq \cdots $, $ E(G_1) \subseteq E(G_2) \subseteq \cdots $
and for each $ m $, $ E(G_m) $ is exactly the set of all edges in $ E(G_{m+1}) $
with both endpoints in $ V(G_m) $. Notice that this allows us to define the endpoints
of each edge $ k \in \mathcal E $ consistently. All examples given in \cite{CE3} are of
this kind.

Following Boglioni Beaulieu et al.\ \cite{CE3}, arrange the edge set $ \mathcal E $ into
a sequence, so that the elements of $ E(G_1) $ come first, followed by the elements of
$ E(G_2) \setminus E(G_1) $, then by $ E(G_3) \bigm\backslash \bigl( E(G_1) \cup E(G_2) \bigr) $
and so on; otherwise, the order does not matter. Without loss of generality, we can just assume
that $ E(G_1) = \{ 1, 2, \ldots, n_1 \} $ and $ E(G_m) \bigm\backslash \bigl(
E(G_1) \cup \cdots \cup E(G_{m-1}) \bigr) = \{ n_{m-1} + 1, n_{m-1} + 2, \ldots, n_m \} $
for $ m = 2, 3, 4, \ldots $\spacefactor=3000{}
Thus, we have obtained a sequence of random variables $ X_1, X_2, X_2, \ldots $, which are
$ K $-tuplewise independent provided that each graph $ G_m $ has girth at least $ K + 1 $.
Letting $ S_n $ be as in \eqref{eq:Sn}, notice that $ S^*_m = S_{n_m} $. Therefore, under the
conditions of Corollary~\ref{co:CEArr}, the sequence $ S_1, S_2, \ldots $ has a subsequence
which converges in law to a non-normal distribution. Hence the sequence $ S_1, S_2, \ldots $
does not converge to a normal distribution. This is what is proved in \cite{CE3}
(under Condition~\ref{cn:l-1}).

However, under some additional conditions, one can do a bit more, showing that the whole
sequence of standardized partial sums actually converges in law.

\begin{lemma}
\label{lm:FillGaps}
Let $ V_1, V_2, \ldots $ be uncorrelated zero-mean random variables with the same
variance $ \sigma^2 $, where $ 0 < \sigma < \infty $. Let
\[
 T_n := \frac{1}{\sigma \sqrt{n}} \sum_{k=1}^n V_k
\]
be their standardized partial sums. Take a sequence $ n_1 < n_2 < \ldots $ of natural
numbers with $ \lim_{m \to \infty} n_{m+1}/n_m = 1 $. Then, if the subsequence
$ T_{n_1}, T_{n_2}, \ldots $ converges in law to a random variable $ T $, the same
is true for the whole sequence $ T_1, T_2, \ldots $
\end{lemma}

\begin{proof}
Letting $ N_n := n_m $ for $ n_m \le n < n_{m+1} $, we find that the sequence
\[
 \frac{1}{\sigma \sqrt{N_n}} \sum_{k=1}^{N_n} V_k \> ; \quad n = 1, 2, 3, \ldots
\]
converges in law to $ T $. Notice that the assumed condition implies that
$ \lim_{n \to \infty} n/N_n = 1 $. Now consider the sequence
\[
 \frac{1}{\sigma \sqrt{N_n}} \sum_{k=N_n+1}^{n} V_k \> ; \quad n = 1, 2, 3, \ldots
\]
and observe that
$ \Var \bigl( \frac{1}{\sigma \sqrt{N_n}} \sum_{k=N_n+1}^{n} V_k \bigr) = \frac{n - N_n}{N_n} $
tends to zero as $ n \to \infty $. By Chebyshev's inequality, the random variables
$ \frac{1}{\sigma \sqrt{N_n}} \sum_{k=N_n+1}^{n} V_k $ then converge in law to zero
as $ n \to \infty $. By Slutsky's theorem, the sequence
$ \frac{1}{\sigma \sqrt{N_n}} \sum_{k=1}^n V_k $ then converges in law to $ T $.
The rest is completed by another part of Slutsky's theorem, recalling that
$ \lim_{n \to \infty} n/N_n = 1 $.
\end{proof}

We can now summarize our observations into the following assertion:

\begin{corollary}
\label{co:CESeq}
Let $ W $ be a real-valued random variable with expectation $ \mu $ and
variance $ \sigma^2 $; assume that $ 0 < \sigma < \infty $. Let $ G_m $,
$ n_m $ and $ \xi^*_m $, $ m = 1, 2, 3, \ldots $, be as above. Suppose that all graphs
$ G_m $ have girth at least $ K + 1 $ and that the random variables
$ \xi^*_m $ converge in law to a probability distribution which is not normal.
Then there exist $ K $-tuplewise independent random variables
$ X_1, X_2, X_3, \ldots $, each of them following the same distribution as $ W $,
such that their standardized partial sums $ S_n $ defined as in \eqref{eq:Sn}
do not converge to a normal distribution. If, in addition, $ n_1 < n_2 < \ldots $
and $ \lim_{m \to \infty} n_{m+1}/n_m = 1 $, then the random variables $ S_n $
converge in law to a probability distribution which is not normal.\qed
\end{corollary}

With the preceding assertion, we are in a position to prove Theorem~\ref{th:Seq3}.

\begin{proof}[Proof of Theorem~\ref{th:Seq3}]
Following Boglioni Beaulieu et al.\ \cite{CE3}, choose $ G_m := K_{m,m} $;
as usual, $ K_{m,m} $ denotes the bipartite graph with vertices
divided into two groups of $ m $ vertices, where two vertices are adjacent if and
only if they belong to different groups. Notice that $ K_{m,m} $ has $ m^2 $ edges and
girth 4 for $ m \ge 2 $. Choosing any $ \ell \in \{ 2, 3, 4, \ldots \} $,
Theorem~4.1 in \cite{CE3} shows that the underlying random variables $ \xi^*_m $
defined as above converge in law to a variance gamma distribution,
which is not normal. Thus, the conditions of Corollary~\ref{co:CESeq} are fulfilled,
proving the result.
\end{proof}

\section{Construction of mixture}
\label{sc:Mix}

It remains to prove Proposition~\ref{pr:MixDiffExp}, which claims that any suitable
probability distribution on the real line can be represented as a suitable mixture
of two distributions.

\begin{proof}[Proof of Proposition~\ref{pr:MixDiffExp}]
Let
$ a := \sup \{ w \in \RR \sth \Pp(W < w) < 1 - \tau \} $
and
$ b := \inf \{ w \in \RR \sth \Pp(W > w) < \tau \} $
be the lower and upper $ (1 - \tau) $-quantile of the random variable $ W $.
Clearly, $ a \le b $, $ \Pp(W < a) \le 1 - \tau \le \Pp(W \le a) $ and
$ \Pp(W > b) \le \tau \le \Pp(W \ge b) $. We now distinguish two cases.

First, if $ a < b $ or $ \Pp(a \le W \le b) = 0 $, then $ \Pp(W \le a) = 1 - \tau $ and
$ \Pp(W \ge b) = \tau $. In this case, the construction is exactly the same as in
\cite{CE2,CE3}: choosing $ U $ and $ V $ to follow the conditional distribution of $ W $ given
$ W \le a $ and $ W \ge b $, respectively, \eqref{eq:MixDiffExp} is immediate.
Moreover, $ \Ee U = \Ee(W \mid W \le a) \le a < b = \Ee(W \mid W \ge b) = \Ee V $.
Finally, if $ W $ has finite variance, that is, if $ \Ee(W^2) < \infty $,
then $ \Ee(U^2) = \Ee(W^2 \mid W \le a) $ and $ \Ee(V^2) = \Ee(W^2 \mid W \ge b) $
are finite, too.

It remains to consider the case where $ a = b $ and $ \Pp(W = a) > 0 $.
Then define the distributions of $ U $ and $ V $ by
\[
 \Pp(U \in C)
 =
 \frac{1}{1 - \tau} \left( \Pp(W \in C, W < a) + \frac{1 - \tau - \Pp(W < a)}{\Pp(W = a)} \, \Pp(W \in C, W = a) \right)
\]
and
\[
 \Pp(V \in C)
 =
 \frac{1}{\tau} \left( \Pp(W \in C, W > a) + \frac{\tau - \Pp(W > a)}{\Pp(W = a)} \, \Pp(W \in C, W = a) \right)
 \, .
\]
A brief calculation shows that the latter two formulas indeed define
probability distributions and that \eqref{eq:MixDiffExp} is fulfilled. Next, we show
that we again have $ \Ee U < \Ee V $. First, observe that both expectations exist with
\[
 \Ee U
 =
 \frac{\Ee \bigl[ W \One(W < a) \bigr] + a \bigl( 1 - \tau - \Pp(W < a) \bigr)}{1 - \tau}
\]
and
\[
 \Ee V
 =
 \frac{\Ee \bigl[ W \One(W > a) \bigr] + a \bigl( \tau - \Pp(W < a) \bigr)}{\tau}
 \, .
\]
Now if $ \Pp(W < a) > 0 $, then $ \Ee(W \mid W < a) < a $. A brief calculation shows that
$ \Ee U < a $ in this case. Similarly, if $ \Pp(W > a) > 0 $, then $ \Ee V > a $.
Since $ W $ is not almost surely constant, at least one of these two cases occurs.
Noting that $ \Ee U \le a \le \Ee V $, we conclude that $ \Ee U < \Ee V $.

Finally, observe that if $ W $ has finite variance, that is, if $ \Ee(W^2) < \infty $,
we also have
\[
 \Ee(U^2)
 =
 \frac{\Ee \bigl[ W^2 \One(W < a) \bigr] + a^2 \bigl( 1 - \tau - \Pp(W < a) \bigr)}{1 - \tau}
 <
 \infty
\]
and
\[
 \Ee(V^2)
 =
 \frac{\Ee \bigl[ W^2 \One(W > a) \bigr] + a^2 \bigl( \tau - \Pp(W < a) \bigr)}{\tau}
 <
 \infty
\]
and the proof is complete.
\end{proof}

\bigskip

\ACKNOS{%
  The author acknowledges financial support from the ARIS (Slovenian Research and
  Innovation Agency, research core funding No.~P1--0448). In addition, the author
  would like to thank Roman Drnovšek for reading the text and valuable suggestions.%
}

\bibliographystyle{siam}
\bibliography{TupleInd}    

\end{document}